\RequirePackage{fix-cm}
\documentclass[smallextended]{svjour3}       
\smartqed  

\usepackage{amsmath, amsfonts, vmargin, enumerate}

\usepackage{amssymb}

\numberwithin{equation}{section}

\newcommand{\real}{{\mathbb R}}
\newcommand{\A}{{\mathcal A}}

\newcommand{\E}{{\mathcal E}}

\newcommand{\M}{{\mathcal M}}
\newcommand{\N}{{\mathcal N}}

\newcommand{\D}{{\mathcal D}}
\newcommand{\8}{\infty}

\newcommand{\be}{\begin{eqnarray*}}
\newcommand{\ee}{\end{eqnarray*}}
\newcommand{\beq}{\begin{equation}}
\newcommand{\eeq}{\end{equation}}
\newcommand{\beqn}{\begin{equation*}}
\newcommand{\eeqn}{\end{equation*}}
\newcommand{\bsp}{\begin{split}}
\newcommand{\esp}{\end{split}}

\begin{document}

\title{{ A Beurling-Blecher-Labuschagne type  theorem
for Haagerup noncommutative\\ $L^p$ spaces}
\thanks{ T.N. Bekjan is partially supported by NSFC grant\\ No.11771372, M. Raikhan is partially supported by  project AP05131557 of the Science Committee of Ministry of Education and Science of the Republic of Kazakhstan.}}

\author{Turdebek N. Bekjan        \and
        Madi Raikhan}

\authorrunning{T. N. Bekjan, Madi Raikhan} 

\institute{T. N. Bekjan \at
             College of Mathematics and Systems Science, Xinjiang
University, Urumqi 830046, China. \\
             \email{bekjant@yahoo.com}
           \and
           M. Raikhan\at
             Astana IT University,   Nur-Sultan 010000, Kazakhstan.\\
  \email{madi.raikhan@astanait.edu.kz} }

\date{Received: date / Accepted: date}

\maketitle

\begin{abstract}
Let $\M$ be a $\sigma$-finite von Neumann algebra, equipped with a
 normal  faithful state $\varphi$, and  let $\mathcal{A}$ be maximal subdiagonal subalgebra of $\mathcal{M}$ and $1\le p<\8$. We prove a Beurling-Blecher-Labuschagne type theorem
for $\A$-invariant subspaces of Haagerup noncommutative $L^p(\A)$ and
 give a characterization of outer operators  in  Haagerup
noncommutative $H^{p}$-spaces associated with $\A$.
\keywords{subdiagonal  algebras,  Beurling's theorem, invariant subspace, outer operator, Haagerup noncommutative $H^{p}$-space}
 \subclass{  46L52 \and  47L05}
\end{abstract}

\section{Introduction}

Arveson introduced his
notion of subdiagonal subalgebras of von Neumann algebras (see \cite{A}), in effect, subdiagonal algebras are the
noncommutative analogue of weak* Dirichlet algebras (for the definition of weak* Dirichlet algebras see \cite{SW}). For the finite and semi-finite case,  most results on
the classical Hardy spaces on the torus have been established in
this noncommutative setting. We refer to \cite{A,BX,B1,B2,BL1,BL2,BL3,BL4,PX,Sa} (see
also \cite{BL4} for more  historical references).  It is natural to  consider the case
of $\sigma$-finite von Neumann algebras. But,  the transition from finite or semifinite  to $\sigma$-finite von Neumann algebras is not trival, need some new techniques and some changes.
For some results for this case, see \cite{BM,JOS,J1,J2,L2,X}.

Let $\M$ be a finite von Neumann and $\A$ be its Arveson's maximal subdiagonal subalgebras. In \cite{BL3}, Blecher and  Labuschagne extended the classical Beurling's theorem  to describe closed $\A$-invariant subspaces in noncommutative space $L^p(\M)$ with $1\le p\le\8$. Sager \cite{Sa} extended the work of Blecher and Labuschagne from a finite von
Neumann algebra to semifinite von Neumann algebras, proved a Beurling-Blecher-Labuschagne theorem for $\A$-invariant spaces of
$L^p(\M)$ when $0 < p \le\8$. The Beurling theorem  has been  generalized  to
the setting of unitarily invariant norms on finite and semifinite von Neumann algebras (see \cite{B2}, \cite{CHS}, \cite{SL}).

When $\A$ is subdiagonal subalgebra of  $\sigma$-finite von Neumann $\M$, Labu-\\ schagne \cite{L2}  showed that a Beurling type theory of invariant subspaces of
noncommutative $H^2$-spaces holds true.  A motivation for this paper is to extend the result in \cite{L2} to the setting of the Haagerup noncommutative
$L^p$-spaces for $1\le p<\8$.

Blecher and  Labuschagne \cite{BL1} studied outer operators of the noncommutative $H^p$-spaces
associated with Arveson's subdiagonal subalgebras. They proved  inner-outer
factorization theorem and characterizations of outer operators for the case $1\le p<\infty$ (for the case $p < 1$,
see \cite{BX}).  In \cite{BL2}, they  extended their generalized inner-outer
factorization theorem in \cite{BL1} and  established characterizations of outer operators that are valid even in the case of
operators with zero determinant. In this paper, we apply Labuschagne's Beurling type theorem
for $\A$-invariant subspaces of Haagerup noncommutative $L^2$-spaces to prove a Blecher-Labuschagne theorem for outer operators in Haagerup
noncommutative  $H^p$-spaces ($1\le p<\8$).

The organization of the paper is as follows. In Section 2, we
give some definitions and  related results of  Haagerup
noncommutative $L^{p}$-spaces and $H^{p}$-spaces.  A Blecher-Labuschagne-Beurling type theorem
for  Haagerup
noncommutative $L^{p}$-spaces is presented in Section 3. In Section 4, we give characterizations of outer operators in Haagerup
noncommutative  $H^{p}$-spaces.

\section{Preliminaries}

Our references  for modular theory  are \cite{PT,T3}, for the Haagerup noncommutative
$L^p$-spaces are \cite{H1,Te} and for the Haagerup noncommutative
$H^p$-spaces are \cite{J1,J2}. Let us recall some basic facts about the Haagerup noncommutative
$L^p$-spaces and the Haagerup noncommutative
$H^p$-spaces, and fix the relevant
notation used throughout this paper.
Throughout this paper $\mathcal{M}$ will always denote  a $\sigma$-finite von Neumann algebra on a complex
Hilbert space $\mathcal{H}$, equipped with a distinguished
 normal  faithful state $\varphi$.
Let $\{\sigma_{t}^\varphi\}_{t\in\real}$
 be the one parameter modular automorphism
group of $\mathcal{M}$ associated with $\varphi$. We denote by
$$
\mathcal{N}=\mathcal{M}\rtimes_{\sigma^\varphi}\real
$$
the crossed product of $\mathcal{M}$ by
$\{\sigma_{t}^\varphi\}_{t\in\mathbb{R}}$. It is well known that
$\mathcal{N}$ is the semi-finite von Newmann algebra acting on the Hilbert space
$L^{2}(\mathbb{R},\mathcal{H}),$ generated by
$$
\left\{ \pi(x):\;x\in\mathcal{M}\right\}\cup\left\{\lambda(s):\;s\in \mathbb{R}\right\},
$$
where  the operator $\pi(x)$ is defined by
$$
(\pi(x)\xi)(t)=\sigma_{-t}^\varphi(x)\xi(t),\qquad \forall\xi\in
L^{2}(\mathbb{R},\mathcal{H}),\quad\forall t\in\mathbb{R},
$$
and the operator $\lambda(s)$ is defined by

$$
(\lambda(s)\xi)(t)=\xi(t-s),\qquad \forall\xi\in
L^{2}(\mathbb{R},\mathcal{H}),\quad\forall t\in\mathbb{R}.
$$
We will identify $\mathcal{M}$ and the subalgebra $\pi(\mathcal{M})$ of $\mathcal{N}$.
The operators $\pi(x)$ and $\lambda(t)$ satisfy
$$
\lambda(t)\pi(x)\lambda(t)^*=\pi(\sigma_{t}^\varphi(x)), \qquad\forall t\in \real, \quad\forall x\in\M.
$$
Then
$$
\sigma_{t}^\varphi(x)=\lambda(t)x\lambda^{\ast}(t), \qquad x\in\mathcal{M},\quad t\in\mathbb{R}.
$$

We denote by $\{\hat{\sigma}_{t}\}_{t\in\mathbb{R}}$ the dual action of $\mathbb{R}$
on $\mathcal{N}$, this is a one parameter automorphism group of $\mathbb{R}$ on
$\mathcal{N},$ implemented by the unitary representation
$\{W_{t}\}_{t\in\mathbb{R}}$ of $\mathbb{R}$ on $L^{2}(\mathbb{R},\mathcal{H}):$
\begin{equation}\label{dual action}
\hat{\sigma}_{t}(x)=W(t)xW^{\ast}(t),\qquad \forall x\in\mathcal{N},\quad\forall t\in\mathbb{R},
\end{equation}
where
$$
W(t)(\xi)(s)=e^{-its}\xi(s),\qquad \forall\xi\in L^{2}(\mathbb{R},\mathcal{H}),\quad\forall
s,t\in\mathbb{R}.
$$
Note that the dual action $\hat{\sigma}_{t}$ is uniquely determined by the
following conditions: for any $x\in\M$ and $s\in\mathbb{R}$,
$$
\hat{\sigma}_{t}(x)=x \quad\mbox{and}\quad
\hat{\sigma}_{t}(\lambda(s))=e^{-ist}\lambda(s),\quad \forall t\in\mathbb{R}.
$$
Hence
$$
\mathcal{M}=\{x\in\mathcal{N}: \; \hat{\sigma}_{t}(x)=x,\;\forall
t\in\mathbb{R}\}.
$$
Let $\tau$ be  the unique  normal semi-finite faithful trace  on $\mathcal{N}$ satisfying
$$
\tau\circ\hat{\sigma}_{t}=e^{-t}\tau,\qquad \forall t\in\mathbb{R}.
$$

Also recall that the dual weight
$\hat{\varphi}$ of our distinguished state $\varphi$ has the
Radon -Nikodym derivative $D$ with respect to $\tau$, which is the unique invertible positive selfadjoint operator on
$L^{2}(\mathbb{R},\mathcal{H}),$ affiliated with $\mathcal{N}$ such that
$$
\hat{\varphi}(x)=\tau(Dx),\qquad x\in \mathcal{N}_{+}.
$$
Recall  that the
regular representation of the above $\lambda(t)$ is given by
$$
\lambda(t)=D^{it},\qquad \forall t\in \mathbb{R}.
$$
 Now, we define
Haagerup noncommutative $L^{p}$-spaces. Let $L^{0}(\mathcal{N},\tau)$ denote
 the topological $\ast$-algebra of all operators on
$L^{2}(\mathbb{R},\mathcal{H})$ measurable with respect to
$(\mathcal{N},\tau)$. Then the Haagerup noncommutative $L^{p}$-spaces,
$0<p\leq\infty$, are defined by
$$
L^{p}(\mathcal{M},\varphi)=\{x\in L^{0}(\mathcal{N},\tau):
\;\hat{\sigma}_{t}(x)=e^{-\frac{t}{p}}x,\;\forall t\in\mathbb{R}\}.
$$
The spaces  $L^{p}(\mathcal{M},\varphi)$ are closed selfadjoint linear subspaces of
$L^{0}(\mathcal{N},\tau)$.  It is not hard to show
that
$$
L^{\infty}(\mathcal{M},\varphi)=\mathcal{M}.
$$
Since for any $\psi\in\mathcal{M}_{\ast}^{+},$ the dual
weight $\hat{\psi}$ has a Radon-Nikodym derivative with respect to
$\tau,$ denoted by $D_{\psi}:$
$$
\hat{\psi}(x)=\tau(D_{\psi}x),\qquad x\in\mathcal{N}_{+}.
$$
Then
$$
D_{\psi}\in L^{0}(\mathcal{N},\tau)
$$
and
$$
\hat{\sigma}_{t}(D_{\psi})=e^{-t}D_{\psi},\qquad\forall
t\in\mathbb{R}.
$$
So
$$
D_{\psi}\in L^{1}(\mathcal{M},\varphi)_{+}.
$$
It is well known that the map $\psi\mapsto D_\psi$
on $\mathcal{M}_{\ast}^{+}$ extends to a linear homeomorphism from $\mathcal{M}_{\ast}$ onto $L^{1}(\mathcal{M},\varphi)$ (equipped with the vector
space topology inherited from $L^0(\N, \tau)$). This permits to transfer the norm  on $\mathcal{M}_{\ast}$ into a norm
on $L^{1}(\mathcal{M},\varphi)$, denoted by $\|\cdot\|$. Moreover, $L^{1}(\mathcal{M},\varphi)$ is equipped with a distinguished contractive
positive linear functional tr,  defined by
$$
tr (D_\psi) = \psi(1), \qquad \psi\in \mathcal{M}_{\ast}.
$$
Therefore, $\|x\|_1= tr (|x|)$ for every $x\in L^{1}(\mathcal{M},\varphi)$.

Let $0 < p < \8$ and $x\in L^0(\N, \tau)$. If $x = u|x|$ is the polar decomposition of $x$, then
$x \in L^p(\M, \varphi)\; \Leftrightarrow\; u \in\M \;\mbox{and}\; |x| \in L^p(\M, \varphi) \;\Leftrightarrow\; u \in\M \;\mbox{and}\; |x|^p \in L^1(\M, \varphi)$.
If we define
$$
\| x\|_{p}=\| |x|^{p}\|_{1}^{\frac{1}{p}},\qquad \forall x\in
L^{p}(\mathcal{M},\varphi),
$$
then for $1\leq p<\infty$ (resp. $0<p<1$),
$$
(L^{p}(\mathcal{M},\varphi),\;\|\cdot\|_p)
$$
is a Banach space (resp. a quasi-Banach space), and
$$
\| x\|_{p}=\| x^{\ast}\|_{p}=\| |x|\|_{p},\qquad \forall x\in
L^{p}(\mathcal{M},\varphi).
$$

It is proved in \cite{H1} and \cite{Te} that  $L^{p}(\mathcal{M},\varphi)$ is independent of
$\varphi$ up to isometry. Hence,  we denote $L^{p}(\mathcal{M},\varphi)$ by
$L^{p}(\mathcal{M})$.

 The usual Holder inequality also holds for the $L^{p}(\mathcal{M})$ spaces. It means that  the
product of $L^0(\N, \tau),\; (x, y) \mapsto xy$, restricts to a contractive bilinear map
$$
L^{p}(\mathcal{M})\times L^{q}(\mathcal{M})\rightarrow L^{r}(\mathcal{M}),
$$
where  $\frac{1}{r}=\frac{1}{p}+\frac{1}{q}$. In particular, if $\frac{1}{p}+\frac{1}{q}=1$, then the bilinear form $(x, y)\mapsto tr (xy)$ defines a duality bracket between
$L^{p}(\mathcal{M})$ and $L^{q}(\mathcal{M})$, for which $L^{q}(\mathcal{M})$ coincides (isometrically) with the dual of $L^{p}(\mathcal{M})$ (if $p\neq\8$).
Moreover, the $tr$ have the following  property:
$$
tr(xy)=tr(yx),\qquad \forall x\in L^{p}(\mathcal{M}),\quad\forall y\in
L^{q}(\mathcal{M}).
$$

Let $0<p\le\infty$. For $K \subset L^{p}(\mathcal{M})$, we denote  the closed linear span of  $K$ in $L^{p}(\mathcal{M})$ by $[K]_{p}$ (relative to the w*-topology for $p =\infty$) and the set  $\{x^{\ast}:\;x\in K\}$ by $J(K)$.

For $0<p<\infty,\; 0\le \eta\le1$, we have that
$$
L^{p}(\mathcal{M})=[D^{\frac{1-\eta}{p}}\M D^{\frac{\eta}{p}}]_p.
$$

 Let $\mathcal{D}$ be a von Neumann
 subalgebra of $\mathcal{M} $ and $\mathcal{E}$ be  a faithful
normal conditional expectation from $\M$ onto  $\D$.

\begin{definition} A w*-closed subalgebra $\mathcal{A}$ of $\mathcal{M}$
is called a subdiagonal subalgebra of $\mathcal{M}$ with respect to
$\mathcal{E}$(or to $\mathcal{D}$)  if
\begin{enumerate}[\rm(i)]

\item $\mathcal{A}+ J(\mathcal{A})$ is w*-dense in  $\mathcal{M}$,

\item $\mathcal{E}(xy)=\mathcal{E}(x)\mathcal{E}(y),\quad \forall\;x,y\in
\mathcal{A},$

\item $\mathcal{A}\cap J(\mathcal{A})=\mathcal{D},$
\end{enumerate}
The algebra $\mathcal{D}$ is  called the diagonal of $\mathcal{A}$.
\end{definition}
In \cite{A}, subdiagonal subalgebras are not assumed
to be w*-weakly closed. Since the weak* closure of an algebra that is subdiagonal with respect to $\E$ will also be subdiagonal with respect to $\E$ (see Remark 2.1.2 in \cite{A}),
 we may  assume that our subdiagonal subalgebras are always w*-weakly closed (the definition as in \cite{J1,J2,X}). Since $\M$ is $\sigma$-finite, we
may take a faithful normal state $\phi$  on $\M$ such that $\phi\circ\E=\phi$.  It is well known (cf. \cite{T3}) that the existence of a (unique)
normal conditional expectation $\E :\M \rightarrow \D$ such that  $\varphi\circ\E=\varphi$  is equivalent to $\sigma_{t}^\varphi(\D)=\D$  for all $t\in\mathbb{R}$.
Hence, in the rest of this paper $\varphi$ always denotes a  normal  faithful state satisfying
 $\varphi\circ\E=\varphi$.

If $\mathcal{A}$ is not properly contained
in any other subalgebra of $\mathcal{M}$ which is a subdiagonal with respect to
$\mathcal{E}$, We call  $\mathcal{A}$ is a maximal subdiagonal subalgebra of $\mathcal{M}$ with
respect to $\mathcal{E}$ (or to $\mathcal{D}$).   Let
$$
\mathcal{A}_{0}=\{x\in \mathcal{A}:\;\mathcal{E}(x)=0\}
$$
Then by \cite[Theorem  2.2.1]{A}, $\mathcal{A}$ is maximal if and only if
$$
 \mathcal{A}=\{x\in
\mathcal{M}:\;\mathcal{E}(yxz)=\mathcal{E}(yxz)=0,\;\forall y\in
\mathcal{A},\;\forall z\in \mathcal{A}_{0}\}.
$$

 It follows from \cite[Theorem 2.4]{JOS} and \cite[Theorem 1.1]{X} (also see \cite[Theorem 1.1]{L2}) that a subdiagonal subalgebra $\A$ of $\mathcal{M}$ with respect to
 $\mathcal{D}$ is maximal if and only if
\begin{equation}\label{maximal}
\sigma_{t}^\varphi(\A)=\A,\qquad \forall t\in\mathbb{R}.
\end{equation}

In this paper $\mathcal{A}$ always denotes a maximal subdiagonal subalgebra in $\mathcal{M}$ with
respect to $\mathcal{E}$.

\begin{definition}\label{def:hp}   For $0<p<\infty$,  we define the Haagerup
noncommutative $H^{p}$-space that

$$
H^{p}(\mathcal{A})=[\mathcal{A}D^{\frac{1}{p}}]_{p},\qquad H^{p}_{0}(\mathcal{A})=
 [\mathcal{A}_{0}D^{\frac{1}{p}}]_{p}.
$$
\end{definition}

If $1\le p<\infty,\; 0\le \eta\le1$, then by \cite[Proposition 2.1]{J2}, we have that
\begin{equation}\label{equalityHp}
H^{p}(\mathcal{A})=[D^{\frac{1-\eta}{p}}\A D^{\frac{\eta}{p}}]_p,\qquad H^{p}_0(\mathcal{A})=[D^{\frac{1-\eta}{p}}\A_0 D^{\frac{\eta}{p}}]_p.
\end{equation}
By \cite[Proposition 2.7]{BM}, we know that
\begin{equation}\label{charecterization A}
\mathcal{A}=\{x\in\mathcal{M}:\;tr(xa)=0,\; \forall a\in H^{1}_{0}(\A)\}.
\end{equation}
It is known that
\begin{equation} \label{equalityLp(D)}
L^{p}(\mathcal{D})=[D^{\frac{1-\eta}{p}}\mathcal{D} D^{\frac{\eta}{p}}]_p,\qquad  \forall p\in [1,\infty),\quad \forall \eta\in [1,0].
\end{equation}
Therefore, if $1\le p,q,r<\8$ and  $\frac{1}{q}+\frac{1}{r}=\frac{1}{p}$, then
\begin{equation}\label{Hp-multiplication-Hr}
[H^q(\A)D^{\frac{1}{r}}]_p=H^p(\A)\quad\mbox{and}\quad
[L^q(\D)D^{\frac{1}{r}}]_p=L^p(\D)
\end{equation}

For $1\le p\le\8$, the conditional expectation $\mathcal{E}$ extends to a contractive
projection  from $L^p(\mathcal{M})$ onto $L^p(\mathcal{D})$. The extension
will be denoted still by $\mathcal{E}$ (see \cite[Proposition 2.3]{JX}).
Let
$$
1\le r,\;p,\;q\le\8,\quad\frac{1}{r}=\frac{1}{p}+\frac{1}{q}.
$$
Then
$$
\E(xy)=\E(x)\E(y),\qquad\forall x\in H^p(\A),\quad\forall y\in H^q(\A).
$$

Let $\mathcal{M}_{a}$ be the family of analytic vectors in $\mathcal{M}.$
Recall that $x\in\mathcal{M}_{a}$ if only if the function $t\mapsto
\sigma_{t}(x)$ extends to an analytic function from $\mathbb{C}$ to
$\mathcal{M}.$  $\mathcal{M}_{a}$ is a w*-dense $\ast$-subalgebra of
$\mathcal{M}$ (cf. \cite{PT}).

The next result is known. For easy reference, we give its proof
(see the proof of Theorem 2.5 in \cite{J1}).
\begin{lemma}\label{analytic} Let $\mathcal{A}_{a}$ and $\D_a$ be respectively the families of analytic
vectors in $\mathcal{A}$ and $\D$.  If $1\le p<\infty$, then:
\begin{enumerate}[\rm(i)]
\item $\mathcal{A}_{a}$ is a w*-dense in $\mathcal{A}$, $(\A_{a})_0$ is a w*-dense in $\A_0$ and $\D_{a}$ is a w*-dense in $\D$,
where $(\A_a)_{0}=\{x\in \A_a:\;\mathcal{E}(x)=0\}$;

\item
$$
D^{\pm\frac{1}{p}}\A_{a}=\A_{a}D^{\pm\frac{1}{p}},\qquad D^{\pm\frac{1}{p}}(\A_{a})_0=(\A_{a})_0D^{\pm\frac{1}{p}},\qquad
D^{\pm\frac{1}{p}}\D_{a}=\D_{a}D^{\pm\frac{1}{p}};
 $$
 \item
$\mathcal{A}_{a}D^{\frac{1}{p}}$ is dense in $H^{p}(\A)$, $(\mathcal{A}_{a})_0D^{\frac{1}{p}}$ is dense in $H^{p}_0(\A)$ and $\D_{a}D^{\frac{1}{p}}$ is dense in $L^{p}(\D)$.
\end{enumerate}
\end{lemma}
\begin{proof} (i) Let $x\in\mathcal{A}$. We define
$$
x_{n}=\sqrt{\frac{n}{\pi}}\int_{\mathbb{R}}e^{-nt^{2}}\sigma_{t}(x)dt.
$$
By \eqref{maximal},  $x_{n}\in\mathcal{A}$. Moreover by
\cite[p. 58]{PT}, $x_{n}\in\mathcal{A}_{a}$ and $x_{n}\rightarrow x$ w*-weakly. Since
$$
\sigma_{t}^\varphi(\A_0)=\A_0,\quad \sigma_{t}^\varphi(\D)=\D,\qquad \forall t\in\mathbb{R}
$$
(see \cite[p. 313]{JOS}), a similar argument works for $\A_0$ and $\D$.

(ii) We prove only the first equivalence. The proofs of the two others are
similar. Let $x\in \mathcal{A}_{a}$. Then
$$
D^{\pm\frac{1}{p}}x=[D^{\pm\frac{1}{p}}xD^{\mp\frac{1}{p}}]D^{\pm\frac{1}{p}}
=[\sigma_{\mp\frac{i}{p}}(x)]D^{\pm\frac{1}{p}}\in
\mathcal{A}_{a}D^{\pm\frac{1}{p}},
$$
whence $D^{\pm\frac{1}{p}}x\subseteq
\mathcal{A}_{a}D^{\pm\frac{1}{p}}$. The inverse
inclusion can be proved in a similar way.

(iii) Let $p'$ be the conjugate index of $p$. If $y\in L^{p'}(\M)$
such that $tr(aD^{\frac{1}{p}}y)=0,\; \forall a\in\mathcal{A}_{a}$, then  by (i),
$$
tr(aD^{\frac{1}{p}}y)=0,\qquad \forall a\in\A,
$$
since $D^{\frac{1}{p}}y\in L^{1}(\M)$.
Hence, by \eqref{equalityHp},
$$
tr(xy)=0,\qquad \forall x\in H^{p}(A)
$$
 By the Hahn-Banach theorem,
$\A_{a}D^{\frac{1}{p}}$ is dense in
$H^{p}(\A)$. Similarly, we can prove the two others.

\end{proof}

\section{$\A$-invariant subspaces of $L^p(\M)$}

We recall that a right (resp. left) $\A$-invariant subspace of $L^p(\M)$, is a closed subspace $K$
of $L^p(\M)$ such that $K\mathcal{A}\subset K$ (resp. $\mathcal{A}K\subset K$).

In the case when von Neumann algebra $\M$ is finite,  for a right $\A$-invariant subspace $K$  of $L^2(\M)$, Blecher and Labuschagne \cite{BL3} defined
{\it the right wandering subspace} of $K$ to be the space $W = K\ominus[K\A_0]_2$;
and they say that $K$ is type 1 if $W$ generates $K$ as an $\A$-module (that is, $K=[W\A]_2$) and
say that $K$ is type 2 if $W = \{0\}$ (also see \cite{NW}, but the last notation conflicts with that of
\cite{NW}, where this class of subspaces is decomposed into two further subclasses which
Nakazi and Watatani call type {\rm II} and type {\rm III}).  If $p\neq2$,  Blecher and Labuschagne \cite{BL3} defined the wandering quotient to
be $K/[K\A_0]_p$, and say that $K$ is type 2 if this is trivial. It turns out that the wandering quotient is
an $L^p(\D)$-module in the sense of
Junge and Sherman (see \cite{JS}), and it is isometric to a canonically defined
subspace of $K$ which can be called the right wandering subspace of $K$. They say that
$K$ is type 1 if this subspace generates $K$ as an $\A$-module. For the case $1 \le p < 2$ (resp. $p > 2$),
they have shown that $K$ is type 1 iff $K\cap L^2(\M)$ (resp. $[K]_2$) is type 1 in the sense
of the $L^2$ case above.

Now, in the case that $\mathcal{M}$ is  a $\sigma$-finite von Neumann algebra.  Recall that if $K$ is a right $\mathcal{A}$-invariant subspace of $L^{2}({\mathcal{M}})$, then
$$
W=K\ominus[K\mathcal{A}_{0}]_{2}
$$
is often called  the right wandering subspace of $K$. We say that  $K$ is  type 1 if $W$ generates $K$ as an
$\mathcal{A}$-module (that is $K=[WA]_{2})$ and   $K$ is  type 2 if $W=\{0\}$ (see \cite{L2}).

\begin{proposition}\label{Lp-multiplication} Let $1\le p,q,r<\8$, and $K$ be a closed subspace of $L^{p}(\mathcal{M})$.
 Suppose $\frac{1}{p}-\frac{1}{r}=\frac{1}{q}$ and $K_r=\{x\in K:\;x D^{-\frac{1}{r}}\in L^q(\M)\}$. If $[K_r]_p=K$, then
$$
[[K_rD^{-\frac{1}{r}}]_q D^{\frac{1}{r}}]_p=K.
$$
\end{proposition}
\begin{proof} (1)  If $x\in [K_rD^{-\frac{1}{r}}]_q$, then there is a sequence $(x_n)\subset K_r$ such that $x_nD^{-\frac{1}{r}}\rightarrow x$ in norm in $L^q(\M)$.  Hence, $x_n\rightarrow xD^{\frac{1}{r}}$ in norm in $L^p(\M)$. It follows that $[K_rD^{-\frac{1}{r}}]_q D^{\frac{1}{r}}\subset K$, and so $[[K_rD^{-\frac{1}{r}}]_q D^{\frac{1}{r}}]_p\subset K$. On the other hand, since $K_r\subset[K_rD^{-\frac{1}{r}}]_q D^{\frac{1}{r}}$, $K=[K_r]_p\subset [[K_rD^{-\frac{1}{r}}]_q D^{\frac{1}{r}}]_p$. Therefore, we obtain the desired result.
\end{proof}

\begin{lemma}\label{invariant space}
 Let $1\le p<\8$, and let $K$ be an $\mathcal{A}$-invariant subspace of
$L^{p}({\mathcal{M}})$.
\begin{enumerate}[\rm(i)]
\item If  $1\le q,r<\8$ and $\frac{1}{p}-\frac{1}{r}=\frac{1}{q}$, then
  $[K_rD^{-\frac{1}{r}}]_q$ is a right $\mathcal{A}$-invariant subspace of
$L^{q}({\mathcal{M}})$,  where $K_r=\{x\in K:\;x D^{-\frac{1}{r}}\in L^q(\M)\}$.
   \item If $1\le q,r<\8$ and $\frac{1}{p}+\frac{1}{r}=\frac{1}{q}$, then $[KD^\frac{1}{r}]_q$ is a right $\mathcal{A}$-invariant subspace of $L^{q}({\mathcal{M}})$.
\end{enumerate}
\end{lemma}
\begin{proof} (i) It is clear that $[K_rD^{-\frac{1}{r}}]_q\subset L^q(\M)$. Using (ii) of Lemma \ref{analytic}, we get that
\begin{equation}\label{A-analytic-invariant}
K_rD^{-\frac{1}{r}}\mathcal{A}_a=K_r\mathcal{A}_aD^{-\frac{1}{r}}.
\end{equation}
On the other hand, for any $a\in\A_a$ and $x\in K_r$, we have that $xa\in K$. By \eqref{A-analytic-invariant}, there is an element  $a'$ of $\A_a$ such that $xaD^{-\frac{1}{r}}=xD^{-\frac{1}{r}}a'$. It follows that $xaD^{-\frac{1}{r}}\in L^q(\M)$, and so $xa\in K_r$. Hence, $K_r\A_a\subset K_r$. From \eqref{A-analytic-invariant} follows that $K_rD^{-\frac{1}{r}}\mathcal{A}_a\subset K_rD^{-\frac{1}{r}}$ and
\begin{equation}\label{A-analytic-invariant-2}
[K_rD^{-\frac{1}{r}}\mathcal{A}_a]_q\subset [K_rD^{-\frac{1}{r}}]_q.
\end{equation}

Now  if $a\in \A$, then by (i) in Lemma \ref{analytic}, we have  a sequence $(a_n)$ in $\A_a$ such that
$a_{n}\rightarrow a$ w*-weakly. Hence,
$$
tr(xD^{-\frac{1}{r}}a_ny)\rightarrow tr(xD^{-\frac{1}{r}}ay),\qquad \forall x\in K_r,\quad\forall y\in L^{q'}(\M),
$$
where $q'$ is the conjugate index of $q$.
Since the weak closure of $K_rD^{-\frac{1}{r}}\mathcal{A}_a$ is equal to $[K_rD^{-\frac{1}{r}}\mathcal{A}_a]_q$,
$$
xD^{-\frac{1}{r}}a\in [K_rD^{-\frac{1}{r}}\mathcal{A}_a]_q.
$$
Using \eqref{A-analytic-invariant-2}, we get
$$
[K_rD^{-\frac{1}{r}}\A]_q\subset [K_rD^{-\frac{1}{r}}]_q.
$$
Therefore,
$$
[K_rD^{-\frac{1}{r}}]_q\A\subset [K_rD^{-\frac{1}{r}}]_q.
$$
(ii) can be proved in a similar way.

\end{proof}

Using  same method as  in the proof of Lemma \ref{invariant space}, we get the following result.
\begin{lemma}\label{closure-subspace-subdiagonal}
 Let $1\le p<\8$, and let $K\subset L^{p}({\mathcal{M}})$.
If  $1\le q,r<\8$ and  $\frac{1}{p}+\frac{1}{r}=\frac{1}{q}$, then
$$
[[K\A]_p D^{\frac{1}{r}}]_q=[KD^\frac{1}{r}\A ]_q,\qquad [[K\A_0]_q D^{\frac{1}{r}}]_q=[KD^\frac{1}{r}\A_0 ]_q
$$
and
$$
[[K\D]_p D^{\frac{1}{r}}]_q=[KD^\frac{1}{r}\D ]_q.
$$
\end{lemma}

\begin{lemma}\label{hp-spaces-carecterization}  Let $1\le p<\8$.
If  $1< q,r<\8$ and  $\frac{1}{p}-\frac{1}{r}=\frac{1}{q}$, then
$$
H^p(\A)D^{-\frac{1}{r}}\cap L^q(\M)= H^q(\A)
\quad\mbox{and}\quad
D^{-\frac{1}{r}}H^p(\A)\cap L^q(\M)= H^q(\A).
$$
\end{lemma}
\begin{proof} Let $x\in H^p(\A)D^{-\frac{1}{r}}\cap L^q(\M)$. Then there is an element $y\in H^p(\A)$ such that $x=yD^{-\frac{1}{r}}$. If  $q'$ (resp. $p'$) is the conjugate index of $q$ (resp. $p$), then $\frac{1}{q'}=\frac{1}{p'} +\frac{1}{r}$. Hence,
$$
tr(xD^{\frac{1}{q'}}a)=tr(yD^{-\frac{1}{r}}D^{\frac{1}{q'}}a)=tr(yD^{\frac{1}{p'}}a)=0, \qquad \forall a\in\A_0.
$$
Using \eqref{equalityHp}, we get $x\bot J(H^{q'}_0(\A))$. By \cite[Corollary 3.4]{J2} (or \cite[(2.13)]{BM}), $x\in H^q(\A)$, and so $ H^p(\A)D^{-\frac{1}{r}}\cap L^q(\M)\subset H^q(\A)$. Conversely, from
$H^q(\A)D^{\frac{1}{r}}\subset H^p(\A)$ it follows that $ H^p(\A)D^{-\frac{1}{r}}\cap L^p(\M)\supset H^q(\A)$. Thus, we obtain the first result. The second result follows analogously.
\end{proof}

\begin{definition}\label{wanderingsubspace}
 Let $1\le p<\8$, and let $K$ be a  right $\mathcal{A}$-invariant subspace of
$L^{p}({\mathcal{M}})$.
\begin{enumerate}[\rm(i)]
  \item  If $1 \le p \le 2,\;\frac{1}{p}-\frac{1}{r}=\frac{1}{2}$  and $W$ is the right wandering subspace of $[KD^{-\frac{1}{r}}\cap L^2(\M)]_2$, we define the right wandering subspace of K  to be the $L^p$-closure of  $WD^{\frac{1}{r}}$.
  \item  If
$2\le p <\8,\;\frac{1}{p}+\frac{1}{r}=\frac{1}{2}$ and $W$ is the right wandering subspace of
$[KD^\frac{1}{r}]_2$, we define the right wandering subspace of K  to be the $L^p$-closure of $W_rD^{-\frac{1}{r}}$,
where $W_r=\{x\in W:\;x D^{-\frac{1}{r}}\in L^p(\M)\}$.
\end{enumerate}

\end{definition}

If $K$
is a right $\mathcal{A}$-invariant subspace of $L^p(\mathcal{M})$, we  say that $K$ is type 1 if the right wandering subspace of $K$  generates $K$ as an
$\mathcal{A}$-module, and $K$ is type 2 if  $1 \le p < 2$ (resp. $p > 2$) and  $K=[K\mathcal{A}_{0}]_p$ (resp. $[KD^{\frac{1}{r}}]_2=[KD^{\frac{1}{r}}\mathcal{A}_{0}]_2$, where $\frac{1}{p}+\frac{1}{r}=\frac{1}{2}$).

To extend the result in \cite{L2} to the setting of the Haagerup noncommutative
$L^p$-spaces ($1\le p<\8$), we will use the column $L^p$-sum studied by Junge and Sherman \cite{JS} to investigate this: If $X$ is a subspace
of $L^p(\M)$, and if $\{X_i : i \in I\}$ is a collection of subspaces of X, which together
densely span $X$, with the property that $X^*_i X_j=\{0\}$ if $i\neq j$, then we say that $X$
is the internal column $L^p$-sum $\oplus_i^{col}X_i$.

\begin{theorem}\label{invariant}
Let $1\le p<2$ and  $K$ be a right
 $\mathcal{A}$-invariant subspace of $L^p(\mathcal{M})$. Suppose $\frac{1}{p}-\frac{1}{r}=\frac{1}{2}$ and $K_r=\{x\in K:\;x D^{-\frac{1}{r}}\in L^2(\M)\}$. If $[K_r]_p=K$, then:
 \begin{enumerate}[\rm(i)]
 \item  $K$ may be written uniquely as an $L^p$-column sum $Z\oplus^{col} [Y\A]_p$, where  $Z$ is a type 2
 right $\mathcal{A}$-invariant subspace of $L^p(\M)$, $Y$ is the right wandering subspace of $K$ such that $Y=[Y\D]_p$ and $J(Y)Y\subset L^{\frac{p}{2}}(\D)$.

\item  If $K \neq\{0\}$ then K is type 1 if and only if $K=\oplus_{i}^{col}u_{i}H^p(\mathcal{A})$,
 for $u_{i}$ partial
isometries with mutually orthogonal ranges and $u_{i}^{*}u_{i}\in \mathcal{D}$.
\item If $K = K_1 \oplus^{col} K_2$  where $K_1$ and $K_2$ are types 2 and 1 respectively, then
the right wandering subspace for $K$ equals the right wandering subspace for
$K_2$.
\item The wandering quotient $K/[K\A_0]p$ is isometrically $\D$-isomorphic to the
right wandering subspace of K.
 \item The wandering subspace $W$ of $K$ is an $L^p(\D)$-module in the sense of
Junge and Sherman.
  \end{enumerate}
 \end{theorem}
\begin{proof} (i)  By Lemma \ref{invariant space}, $K'=[K_rD^{-\frac{1}{r}}]_2$  is a right $\mathcal{A}$-invariant subspace of $L^{2}({\mathcal{M}})$. Using Theorem 2.3 and 2.8 in \cite{L2}, we have that
$$
K'=Z'\oplus^{col} [Y'\A]_2,
$$
 where  $Z'$ is a type 2
 right $\mathcal{A}$-invariant subspace of $L^2(\M)$ and  $Y'$ is the right wandering subspace of $K'$ with $Y'=[Y'\D]_2$ and $J(Y')Y'\subset L^{1}(\D)$. Let $Z=[Z'D^{\frac{1}{r}}]_p$ and $Y=[Y'D^{\frac{1}{r}}]_p$. By Lemma \ref{invariant space} and Definition \ref{wanderingsubspace}, $Z$ is a right $\mathcal{A}$-invariant subspaces of $L^p(\M)$ and $Y$ is the right wandering subspace of $K$. Using Lemma \ref{closure-subspace-subdiagonal}, we know that $[[Y'\A]_2 D^{\frac{1}{r}}]_p=[Y\A]_p $. For any $x\in Z', y\in [Y'\A]_2$, we have that $x^*y=0$, and so
$$
D^{\frac{1}{r}}x^*yD^{\frac{1}{r}}=0.
$$
Hence, $J(Z)[Y\A]_p=\{0\}$. On the other hand, by Proposition \ref{Lp-multiplication}, $K=[K'D^{\frac{1}{r}}]_p$. Therefore,
$$
K=Z\oplus^{col} [Y\A]_p.
$$
Since $Z'=[Z'\A_0]_2,\;Y'=[Y'\D]_2$, by Lemma \ref{closure-subspace-subdiagonal},
$$
Z=[Z'D^{\frac{1}{r}}]_p=[[Z'\A_0]_2D^{\frac{1}{r}}]_p=[Z'\A_0D^{\frac{1}{r}}]_p=[Z'D^{\frac{1}{r}}\A_0]_p=[Z\A_0]_p
$$
and
$$
Y=[Y'D^{\frac{1}{r}}]_p=[[Y'\D]_2D^{\frac{1}{r}}]_p=[Y'\D D^{\frac{1}{r}}]_p=[Y'D^{\frac{1}{r}}\D]_p=[Y\D]_p.
$$
Since
$$
J(Y'D^{\frac{1}{r}})Y'D^{\frac{1}{r}}=D^{\frac{1}{r}}J(Y')Y'D^{\frac{1}{r}}\subset D^{\frac{1}{r}}L^1(\D)D^{\frac{1}{r}}\subset L^{\frac{p}{2}}(\D),
$$
 it follows that $J(Y)Y\subset L^{\frac{p}{2}}(\D)$.

Now we prove the uniqueness.  Suppose that $Z_1$ is a type 2
 right $\mathcal{A}$-invariant subspace of $L^p(\M)$ and $Y_1$ is the right wandering subspace of $K$ such that
 $$
 K=Z_1\oplus^{col} [Y_1\A]_p\quad\mbox{and}\quad Y_1=[Y_1\D]_p.
 $$
Since $Y_1$ is the right wandering subspace of $K$, by Definition \ref{wanderingsubspace}, $Y_1=[Y_1'D^{\frac{1}{r}}]_p$, where
$Y_1'$ is the right wandering subspace of $[KD^{-\frac{1}{r}}\cap L^2(\M)]_2=[K_rD^{-\frac{1}{r}}]_2$. By by the uniqueness
assertion in Theorem 2.3 of \cite{L2},  $Y'=Y'_1$. It follows that  $Y_1=Y$. From $ K=Z_1\oplus^{col} [Y\A]_p=Z\oplus^{col} [Y\A]_p$, we obtain that  $Z_1=Z$.

(ii)  Let $K \neq\{0\}$ and $K$ is type 1. From the proof of (1), we know that  $[K_rD^{-\frac{1}{r}}]_2$  is type 1. So, by \cite[(ii) of Theorem 2.8]{L2}, there are partial
isometries $u_i$  with mutually orthogonal ranges such that $u_i^*u_i \in\D$,
$$
[K_rD^{-\frac{1}{r}}]_2 = \oplus^{col}_i u_i H_2(\A).
$$
Using Proposition \ref{Lp-multiplication} and \eqref{equalityHp}, we get
$$
\begin{array}{rl}
K&=[[K_rD^{-\frac{1}{r}}]_2D^{\frac{1}{r}}]_p=\oplus^{col}_i [u_i H_2(\A)D^{\frac{1}{r}}]_p\\
&=\oplus^{col}_i u_i[ H_2(\A)D^{\frac{1}{r}}]_p=\oplus^{col}_i u_i H_p(\A)
\end{array}
$$

Conversely, let for
$u_i$ as above,
$$
K=\oplus_{i}^{col}u_{i}H^p(\mathcal{A}).
$$
By Lemma \ref{hp-spaces-carecterization}, $[H^p(\A)D^{-\frac{1}{r}}\cap L^2(\M)]_2=H^2(\mathcal{A})$. Hence,
$$
[K_rD^{-\frac{1}{r}}]_2=\oplus_{i}^{col}u_{i}[H^p(\A)D^{-\frac{1}{r}}\cap L^2(\M)]_2
=\oplus_{i}^{col}u_{i}H^2(\mathcal{A}).
$$
 So
 $$
 [K_rD^{-\frac{1}{r}}\A_0]_2=\oplus_{i}^{col}u_{i}H^2_0(\mathcal{A}).
$$
Hence, the right wandering subspace $W$ of $[K_rD^{-\frac{1}{r}}]_2$  satisfies
$$
W=\oplus_{i}^{col}u_{i}L^2(\D).
$$
By Definition \ref{wanderingsubspace} and \eqref{Hp-multiplication-Hr}, $\oplus_{i}^{col}u_{i}L^p(\D)$ is the right wandering subspace of $K$. Since
$$
[\oplus_{i}^{col}u_{i}L^p(\D)\A]_p=\oplus_{i}^{col}u_{i}H^p(\A)=K,
$$
 $K$ is type 1.

(iii)
Set $K_1^{(r)}=\{x\in K_1:\;x D^{-\frac{1}{r}}\in L^2(\M)\}$ and $K_2^{(r)}=\{x\in K_2:\;x D^{-\frac{1}{r}}\in L^2(\M)\}$. If $x\in K_r$, then there exist $z\in K_1$ and $y\in K_2$ such that
$x=z+y$  and $z^*y=0$. It follows that $|xD^{-\frac{1}{r}}|^2=|zD^{-\frac{1}{r}}|^2+|yD^{-\frac{1}{r}}|^2$,  and so $|xD^{-\frac{1}{r}}|\ge|zD^{-\frac{1}{r}}|$, $|xD^{-\frac{1}{r}}|\ge|yD^{-\frac{1}{r}}|$. Since $x D^{-\frac{1}{r}}\in L^2(\M)\subset L^0(\N)$, we get $zD^{-\frac{1}{r}},\:  yD^{-\frac{1}{r}}\in L^0(\N)$. On the other hand,
$$
\hat{\sigma}_{t}(D^{\frac{1}{r}})=e^{-\frac{t}{r}}D^{\frac{1}{r}},\qquad\forall t\in\mathbb{R}.
$$
Hence,
$$
1=\hat{\sigma}_{t}(D^{-\frac{1}{r}}D^{\frac{1}{r}})=e^{-\frac{t}{r}}D^{\frac{1}{r}}\hat{\sigma}_{t}(D^{-\frac{1}{r}}),\qquad\forall t\in\mathbb{R},
$$
so that
$$
\hat{\sigma}_{t}(D^{-\frac{1}{r}})=e^{\frac{t}{r}}D^{-\frac{1}{r}},\qquad\forall t\in\mathbb{R}.
$$
Moreover,
$$
\begin{array}{l}
\hat{\sigma}_{t}(zD^{-\frac{1}{r}})=\hat{\sigma}_{t}(z)\hat{\sigma}_{t}(D^{-\frac{1}{r}})=e^{-\frac{t}{p}+\frac{t}{r}}zD^{-\frac{1}{r}}=e^{-\frac{t}{2}}zD^{-\frac{1}{r}}\\
\hat{\sigma}_{t}(yD^{-\frac{1}{r}})=\hat{\sigma}_{t}(y)\hat{\sigma}_{t}(D^{-\frac{1}{r}})=e^{-\frac{t}{p}+\frac{t}{r}}yD^{-\frac{1}{r}}=e^{-\frac{t}{2}}yD^{-\frac{1}{r}},\qquad\forall t\in\mathbb{R}.
\end{array}
$$
Thus  $zD^{-\frac{1}{r}},\:  yD^{-\frac{1}{r}}\in L^{2}(\mathcal{M})$, i.e., $z\in K_1^{(r)}$ and $y\in K_2^{(r)}$.

Next, we prove that $[K_1^{(r)}]_p=K_1$. To this end let $P: K\rightarrow K_1$ be the projection operator. From the above, we know that $P(K_r)\subset K_1^{(r)}$. If $a\in K_1$, then $a\in K$. Since $[K_r]_p=K$, there exists a sequence $(a_n)\subset K_r$ such that $a_n\rightarrow a$. Hence $P(a_n)\rightarrow P(a)=a$. It follows that $a\in[K_1^{(r)}]_p$, Therefore, $[K_1^{(r)}]_p=K_1$. Similarly,
$[K_2^{(r)}]_p=K_2$.

 $[K_rD^{-\frac{1}{r}}]_2$  is a right $\mathcal{A}$-invariant subspace of $L^{2}({\mathcal{M}})$ and
$$
[K_rD^{-\frac{1}{r}}]_2=[K_1^{(r)}D^{-\frac{1}{r}}]_2 \oplus^{col} [K_2^{(r)}D^{-\frac{1}{r}}]_2
$$
 From the proof of (1), it follows that $[K_1^{(r)}D^{-\frac{1}{r}}]_2$ and $[K_2^{(r)}D^{-\frac{1}{r}}]_2$  are types 2 and 1 respectively. By \cite[Proposition 2.7]{L2}, the right wandering subspace for $[K_rD^{-\frac{1}{r}}]_2$ equals the right wandering subspace for $[K_2^{(r)}D^{-\frac{1}{r}}]_2$. By Definition \ref{wanderingsubspace}, we obtain  the desired result.

(iv) By (i), (ii) and (iii), we get that
$$
K=Z\oplus_{i}^{col}u_{i}H^p(\mathcal{A}),
$$
where $Z$ is a type 2, and
 $u_{i}$ are partial
isometries with mutually orthogonal ranges such that $u_{i}^{*}u_{i}\in \mathcal{D}$ and $\oplus_{i}^{col}u_{i}L^p(\D)$  is the right wandering subspace of $K$. Using the properties of $\E$, similar to the proof (2) of Theorem 4.5 in \cite{BL3}, we prove the desired result. We omit the details.

(v) Since $J(W)W\subset L^{\frac{p}{2}}(\D)$, $W$ is a right $L^p(\D)$-module with inner product $\langle\xi,\eta\rangle=\xi^*\eta$ (see \cite[Definition 3.3]{JS}).
\end{proof}

\begin{lemma}\label{wandering-2}
Let $2< p<\8,\;\frac{1}{p}+\frac{1}{r}=\frac{1}{2}\;(r>2)$ and  $K$ be a right
$\mathcal{A}$-invariant subspace of $L^p(\mathcal{M})$.  If $Y$ is the right wandering subspace of $[KD^\frac{1}{r}]_2$, then $[Y_r]_2=Y$,
 where $Y_r=\{x\in Y:\;x D^{-\frac{1}{r}}\in L^p(\M)\}$.
\end{lemma}
\begin{proof} Let $K'=[KD^\frac{1}{r}]_2$. Then $K'=[K'\A_0]_2\oplus Y$. By \cite[Theorem 2.3 and 2.8]{L2}, $Y=\oplus_{i}^{col}u_{i}L^2(\D)$ where
 $u_{i}$ are partial
isometries with mutually orthogonal ranges such that $u_{i}^{*}u_{i}\in \mathcal{D}$.
Since $\oplus_{i}^{col}u_{i}L^p(\D)D^\frac{1}{r}\subset Y_r$, using \eqref{Hp-multiplication-Hr}, we get $[Y_r]_2=Y$.
\end{proof}

Similar to Theorem \ref{invariant}, we have the following result.

\begin{theorem}\label{invariant-2}
Let $2< p<\8,\;\frac{1}{p}+\frac{1}{r}=\frac{1}{2}$ and  $K$ be a right
 $\mathcal{A}$-invariant subspace of $L^p(\mathcal{M})$.  If $K=[[KD^{\frac{1}{r}}]_2D^{-\frac{1}{r}}\cap L^p(\M)]_p$, then:
 \begin{enumerate}[\rm(i)]
 \item  $K$ may be written uniquely as an $L^p$-column sum $Z\oplus^{col} [Y\A]_p$, where  $Z$ is a type 2
 right $\mathcal{A}$-invariant subspace of $L^p(\M)$, $Y$ is the right wandering subspace of $K$ such that $Y=[Y\D]_p$ and $J(Y)Y\subset L^{\frac{p}{2}}(\D)$.

\item  If $K \neq\{0\}$ then K is type 1 if and only if $K=\oplus_{i}^{col}u_{i}H^p(\mathcal{A})$,
 for $u_{i}$ partial
isometries with mutually orthogonal ranges and $u_{i}^{*}u_{i}\in \mathcal{D}$.
\item If $K = K_1 \oplus^{col} K_2$  where $K_1$ and $K_2$ are types 2 and 1 respectively, then
the right wandering subspace for $K$ equals the right wandering subspace for
$K_2$.
\item The wandering quotient $K/[K\A_0]p$ is isometrically $\D$-isomorphic to the
right wandering subspace of K.
 \item The wandering subspace $W$ of $K$ is an $L^p(\D)$-module in the sense of
Junge and Sherman.
  \end{enumerate}
 \end{theorem}
\begin{proof} (i)  By Lemma \ref{invariant space}, $K'=[KD^{\frac{1}{r}}]_2$  is a right $\mathcal{A}$-invariant subspace of $L^{2}({\mathcal{M}})$. Using Theorem 2.3 and 2.8 in \cite{L2}, we have that
$$
K'=Z'\oplus^{col} [Y'\A]_2,
$$
 where  $Z'$ is a type 2
 right $\mathcal{A}$-invariant subspace of $L^2(\M)$ and  $Y'$ is the right wandering subspace of $K'$ with $Y'=[Y'\D]_2$ and $J(Y')Y'\subset L^{1}(\D)$.
 For simplicity, we set
 $$
 \begin{array}{rl}
 &K_r=\{x\in K':\;x D^{-\frac{1}{r}}\in L^p(\M)\},\\
 &Z_r=\{x\in Z':\;x D^{-\frac{1}{r}}\in L^p(\M)\},\\
 &Y_r=\{x\in Y':\;x D^{-\frac{1}{r}}\in L^p(\M)\},\\
  &X'=[Y'\A]_2\;\mbox{and}\;
 X_r=\{x\in X':\;x D^{-\frac{1}{r}}\in L^p(\M)\}.
 \end{array}
 $$
 Let $Z=[Z_rD^{-\frac{1}{r}}]_p$ and $Y=[Y_rD^{-\frac{1}{r}}]_p$. By Lemma \ref{invariant space} and Definition \ref{wanderingsubspace}, $Z$ is a right $\mathcal{A}$-invariant subspaces of $L^p(\M)$ and $Y$ is the right wandering subspace of $K$. We notice that $K=[[KD^{\frac{1}{r}}]_2D^{-\frac{1}{r}}\cap L^p(\M)]_p$ implies that $K=[K_rD^{-\frac{1}{r}}]_p$.

 Since $KD^{\frac{1}{r}}\subset K_r$, we get $[K_r]_2=K'$.  We use same method as in the proof of (iii) of Theorem \ref{invariant} to obtain that $Z'=[Z_r]_2$, $X'=[X_r]_2$ and \begin{equation}\label{eq:r-col}
 K_r=Z_r\oplus^{col}X_r.
\end{equation}
We have that
$$
[ZD^{\frac{1}{r}}]_2=[[Z_rD^{-\frac{1}{r}}]_pD^{\frac{1}{r}}]_2 =[Z_rD^{-\frac{1}{r}}D^{\frac{1}{r}}]_2=[Z_r]_2=Z'.
$$
Hence,
$$
[ZD^{\frac{1}{r}}\A_0]_2=[[ZD^{\frac{1}{r}}]_2\A_0]_2 =[Z'\A_0]_2=Z'=[ZD^{\frac{1}{r}}]_2,
$$
i.e.,  $Z$ is a type 2
 right $\mathcal{A}$-invariant subspace of $L^p(\M)$. By Lemma \ref{analytic}, we have that  $Y_r\D_a\subset Y_r$,
$$
 Y_rD^{-\frac{1}{r}}\subset Y_rD^{-\frac{1}{r}}\D_a= Y_r\D_aD^{-\frac{1}{r}}\subset Y_rD^{-\frac{1}{r}}
$$
 and $[Y_rD^{-\frac{1}{r}}\D_a]_p=[Y_rD^{-\frac{1}{r}}\D]_p$. Therefore,
it follows that
$$
 Y=[Y_rD^{-\frac{1}{r}}\D]_p=[[Y_rD^{-\frac{1}{r}}]_p\D]_p=[Y\D]_p.
$$
Since
$$
J(Y_rD^{-\frac{1}{r}})Y_rD^{-\frac{1}{r}}=D^{-\frac{1}{r}}J(Y_r)Y_rD^{-\frac{1}{r}}\subset D^{-\frac{1}{r}}L^1(\D)D^{-\frac{1}{r}}\subset L^{\frac{p}{2}}(\D),
$$
 we deduce that $J(Y)Y\subset L^{\frac{p}{2}}(\D)$.

Now we prove that $$
K=Z\oplus^{col} [Y\A]_p.
$$
By \cite[Theorem 2.8]{L2}, there are partial
isometries $u_i$ with mutually orthogonal ranges such that  $|u_i| \in\D$,
$$
X'= \oplus^{col}_i u_i H_2(\A)\quad\mbox{and}\quad Y'= \oplus^{col}_i u_i L_2(\D).
$$
Using  Lemma \ref{hp-spaces-carecterization}, we get that
$$
X_rD^{-\frac{1}{r}}=\oplus_{i}^{col}u_{i}(H^2(\A)D^{-\frac{1}{r}}\cap L^p(\M))
=\oplus_{i}^{col}u_{i}H^p(\mathcal{A}).
$$
and
$$
Y_rD^{-\frac{1}{r}}=\oplus_{i}^{col}u_{i}(L^2(\D)D^{-\frac{1}{r}}\cap L^p(\M))
=\oplus_{i}^{col}u_{i}L^p(\D).
$$
 So, it follows that  $[X_rD^{-\frac{1}{r}}]_p=[Y\A]_p$.

We claim that $K_rD^{-\frac{1}{r}}$ is closed. Indeed, if $x\in [K_rD^{-\frac{1}{r}}]_p$, then there is a sequence $(y_n)$ in $K_r$ such that $y_nD^{-\frac{1}{r}}\rightarrow x$ in norm in $L^p(\M)$. It follows that $y_n\rightarrow xD^{\frac{1}{r}}$ in norm in $L^2(\M)$. Set $y=xD^{\frac{1}{r}}$. It is clear that $y\in K_r$. Hence, $x=yD^{-\frac{1}{r}}\in K_rD^{-\frac{1}{r}}$, i.e., $K_rD^{-\frac{1}{r}}$ is closed. Similarly, we can prove that
$Z_rD^{-\frac{1}{r}}$ and $X_rD^{-\frac{1}{r}}$ are closed.
Thus
$$
K= K_rD^{-\frac{1}{r}},\quad Z= Z_rD^{-\frac{1}{r}}\quad \mbox{and}\quad [Y\A]_p= X_rD^{-\frac{1}{r}}.
$$
Applying \eqref{eq:r-col}, we obtain that $K=Z\oplus^{col} [Y\A]_p$.
The remainder of the proof can be done the same way as in the proof of Theorem \ref{invariant}.
\end{proof}

\begin{remark}
 Let $1\le p<\8$ and  $K$ be a right
 $\mathcal{A}$-invariant subspace of $L^{p}(\mathcal{M})$. In general, if $1\le p<2$ and $\frac{1}{p}-\frac{1}{r}=\frac{1}{2}$, then  $[K_r]_p\subset K$; if  $2<p<\8$ and $\frac{1}{p}+\frac{1}{r}=\frac{1}{2}$, then $K\subset[[KD^{\frac{1}{r}}]_2D^{-\frac{1}{r}}\cap L^p(\M)]_p$. It is unknown at the time of this writing whether for the general case,   the results in Theorem \ref{invariant} and \ref{invariant-2} are hold.
 \end{remark}

We use same method as in the proof of \cite[Proposition 2.4]{L2} to obtain
the following result,  we give its proof.

\begin{proposition}\label{cyclic and separating-L2}
Let $K$ is a right
 $\mathcal{A}$-invariant subspace of $L^2(\mathcal{M})$, and let $W$ be the right
wandering subspace of $K$. If $W$ has a cyclic and separating vector for the $\D$-action, then there is an isometry $u\in \mathcal{M}$ such that
$W=uL^2(\D)$.
\end{proposition}
\begin{proof} By an adaption of an argument from \cite{JS} (see p.13) there exists
an isometric $\D$-module isomorphism $\psi: L^2(\D) \rightarrow W$. Let $h=\psi(D^{\frac{1}{2}})\in W$. Then
$$
tr(d^*h^*hd)=\|\psi(D^{\frac{1}{2}}d)\|^2_2= tr(d^*Dd),\qquad \forall d \in \D.
$$
By \cite[(5) of Theorem 2.3]{L2},
$h^*h\in L^1(\D)$, and so $h^*h = D$. Hence there exists an isometry $u$ with
initial projection $1$ such that $h = uD^{\frac{1}{2}}$. Since $\psi$ is $\D$-module map,
we have that
$$
\psi(D^{\frac{1}{2}}d)=\psi(D^{\frac{1}{2}})d = uD^{\frac{1}{2}}d,\qquad \forall d\in\D.
$$
Since $L^2(\D)=[D^{\frac{1}{2}}\D]$, it follows that  $\psi(L^2(\D))= uL^2(D)$. Thus $W=uL^2(\D)$ and $u^{\ast}u=1$.
\end{proof}
Similar to  Proposition \ref{cyclic and separating-L2}, we have the following result.
\begin{proposition}\label{cyclic and separating-L2-1}
Let $K$ is a left
 $\mathcal{A}$-invariant subspace of $L^2(\mathcal{M})$, and let $W$ be the left
wandering subspace of $K$. If $W$ has a cyclic and separating vector for the $\D$-action, then there is a partial isometry $v\in \mathcal{M}$ such that $vv^*=1$ and
$W=L^2(\D)v$.
\end{proposition}
\section{ Outer operators of $H^p(\A)$}

In the case when von Neumann algebra $\M$ is finite, from the Beurling-Blecher-Labuschagne theorem follows a generalized `inner-outer' factorization.
Let $x\in L^p(\M)\;(1\le p\le\8)$ and $K = [x\A]_p$. If the
right-wandering subspace of $K$ (respectively right-wandering quotient of
$K$) has a nonzero separating and cyclic vector for the right action of $\D$,  then $x$ is of the
form $x = uh$ for some some outer operator  $h\in H^p(\A)$ and a unitary $u\in \M $ (see the lines before the Closing
remark of \cite{BL3}). For more details on  outer operators  we refer to
\cite{BX,BL1,BL2}.

In this section, we consider outer operators in the case that $\mathcal{M}$ is  a $\sigma$-finite von Neumann algebra. Similar to the finite case, we define the outer operators as following.

\begin{definition}
 Let $0<p\le\8$. An operator $h\in H^p(\A)$ is called {\it  a left  outer operator},
{\it a right  outer operator} or {\it a bilaterally  outer operator} according to
$[h\A]_p=H^p(\A)$, $[\A h]_p=H^p(\A)$ or $[\A h\A]_p=H^p(\A)$.
 \end{definition}

  \begin{proposition}\label{prop:bilaterally}
 Let $1\le p<\8$, and let  $h\in H^p(\A)$. The following are equivalent:
 \begin{enumerate}[\rm(i)]
   \item $h$ is a  bilaterally  outer operator;
   \item $\mathcal{E}(h)$ is a  bilaterally  outer operator in $L^p(\mathcal{D})$ and $[\A h\A_0]_p = [\A_0h\A]_p = H^p_0(\mathcal{A})$;
 \item $\mathcal{E}(h)$ is a  bilaterally  outer operator in $L^p(\mathcal{D})$ and $\mathcal{E}(h)-h \in [\A h\A_0]_p = [\A_0h\A]_p $.
 \end{enumerate}
  \end{proposition}
\begin{proof} (i) $\Rightarrow$ (ii). If $h$ is a bilaterally  outer operator,  then for $D^\frac{1}{p}$ there exist two sequence $(a_n),\; (b_n)\subset\A$ such that
\begin{equation}\label{eq:p-convergence}
\|a_nhb_n-D^\frac{1}{p}\|_p \rightarrow 0\; \mbox{as}\;n\rightarrow\8.
\end{equation}
 By continuity of $\E$, we get
 $$
 \|\E(a_n)\E(h)\E(b_n)-D^\frac{1}{p}\|_p\rightarrow 0\quad\mbox{as}\quad n\rightarrow\8.
 $$
  Hence, by \eqref{equalityLp(D)}, we have that
$$
L^{p}(\D)=[D^\frac{1}{p}\D]_p\subset[\D \E(h)\D]_p\subset L^{p}(\D).
$$
So, $\mathcal{E}(h)$ is a  bilaterally  outer operator in $L^p(\mathcal{D})$. Using \eqref{equalityHp} and \eqref{eq:p-convergence}, we deduce that
$$
[\A h\A_0]_p = [\A_0h\A]_p = H^p_0(\mathcal{A}).
$$

(ii) $\Rightarrow$ (iii) is trivial.

(iii) $\Rightarrow$ (i). It is clear that
$$
D^\frac{1}{p}\in[\D \E(h)\D]_p\subset[\A \E(h)\A]_p
$$
and $h\in [\A h\A]_p$.
Hence, $\E(h)=(\E(h)-h)+h\in [\A h\A]_p$. It follows that
$D^\frac{1}{p}\in[\A h\A]_p$. By \eqref{equalityHp},
we obtain that $H^p(\A)=[\A h\A]_p$.
\end{proof}

Similar to Proposition \ref{prop:bilaterally}, we have the following result.

\begin{proposition}\label{prop:left-right}
 Let $1\le p<\8$, and let  $h\in H^p(\A)$. The following are equivalent:
 \begin{enumerate}[\rm(i)]
   \item $h$ is a left  outer operator (resp. a right  outer operator);
   \item $\mathcal{E}(h)$ is a  left  outer operator (resp. a right  outer operator) in $L^p(\mathcal{D})$ and $[h\A_0]_p = H^p_0(\mathcal{A})$ (resp. $[\A_0 h]_p= H^p_0(\mathcal{A})$);
 \item  $\mathcal{E}(h)$ is a left  outer operator (resp., a right  outer operator) in $L^p(\mathcal{D})$ and $\mathcal{E}(h)-h \in [h\A_0]_p$ (resp. $\mathcal{E}(h)-h \in [\A_0 h]_p$).
 \end{enumerate}
  \end{proposition}

\begin{proposition}\label{prop:left-right-outer-lp}
 Let $1\le p<\8$. If  $h\in H^p(\A)$  is a left  outer operator (resp. a right  outer operator), then
 $\mathcal{E}(h)$ and $h$ are  left  outer operator (resp. a right  outer operator) in $L^p(\M)$.
  \end{proposition}
 \begin{proof} Let $h\in H^p(\A)$  be a left  outer operator. From Proposition \ref{prop:left-right} it follows that $[\E(h)\D]_p=L^p(\D)$.
Since $D^{\frac{1}{p}}\in L^p(\D)=[\E(h)\D]_p $, there is a sequence $(d_n)$ in $\D$ such that $\E(h)d_n\rightarrow D^{\frac{1}{p}}$ in norm in $L^p(\M)$. Therefore, $[\E(h)\M]_p=L^p(\M)$.

 Notice that $\E(h)\in H^p(\A)=[h\A]_p$. It follows that there is a sequence $(a_n)$ in $\A$ such that $ha_n\rightarrow \E(h)$, and so $[h\M]_p=L^p(\M)$. The
alternative claim follows analogously.
 \end{proof}

We will keep all previous notations throughout this section. If $h$  is a left  outer operator and it is also
a right  outer operator, then  we  call $h$ is {\it an outer operator}.

\begin{lemma}\label{outer-polar}
Let $0< p<\8$.
\begin{enumerate}[\rm(i)]
\item If $ h\in H^p(\A)$ is  an outer operator in $H^p(\mathcal{A})$ and $h=u|h|$ is the polar decomposition of $h$, then $u$ is a unitary.
\item  If $d\in L^p(\D)$ is  an outer operator in $L^p(\D)$ and $d=v|d|$ is the polar decomposition of $d$, then $v$ is a unitary in $\D$.
\end{enumerate}
\end{lemma}

\begin{proof} (i) Since $h$ is a left  outer operator,  there exists a sequence $(a_n)\subset\A$
such that $ha_n\rightarrow D^\frac{1}{p}$ in norm in $L^p(\M)$. Let $l(h)$ be the left support projection of $h$.
Then $l(h)^\bot ha_n\rightarrow l(h)^\bot D^\frac{1}{p}$ in norm in $L^p(\M)$. On the other hand, $l(h)^\bot ha_n = 0$ for all $n$, and so
 $l(h)^\bot D^\frac{1}{p}=0$. Since $D^\frac{1}{p}$ is invertible, $l(h)^\bot=0$.  Hence, $h$ must have dense range, i.e., $uu^*=l(h)=1$. Similarly, from the fact that $h$ is a right  outer operator, we obtain that $u^*u=r(h)=1$, where $r(h)$ is the right support projection of $h$. Thus $u$ is a unitary.

 (ii) The proof is similar to the proof of (i).
\end{proof}
\begin{theorem}\label{outer property-Lp(D)}
 Let $1\le p<\8$, and let  $d\in L^p(\D)$. The following are equivalent:
 \begin{enumerate}[\rm(i)]
   \item $d$ is an  outer operator in $L^p(\D)$;
    \item $d$ is an  outer operator in $H^p(\A)$;
   \item The left and right
support projections of $d$ are 1;
 \item   $d$  is an  outer operator  in $L^p(\M)$.
 \end{enumerate}
\end{theorem}
\begin{proof} (i) $\Rightarrow$ (ii)  Since $D^\frac{1}{p}\in L^p(\D)=[d\D]_p=[\D d]_p$,  there are sequence $(a_n)$ and $(b_n)$  in $\D\subset\A$ such that $da_n\rightarrow D^\frac{1}{p}$
and $b_n d\rightarrow D^\frac{1}{p}$. Hence, $[d\A]_p=[\A d]_p=H^p(\A)$.

(ii) $\Rightarrow$ (iii) is follows from the proof of Lemma \ref{outer-polar}.

(iii) $\Rightarrow$ (iv). First we  prove $d$ is a left outer operator in $L^p(\M)$.  Let $p'$ be the conjugate index of $p$. If $x\in L^{p'}(\M)$ such that  $tr(xdz)=$ for all $z\in\M$, then  $xd=0$. Hence,
$x=xdd^{-1}=0$, and so $[d\M]_p=L^p(\M)$.  Using the same method, we can prove that $d$ is a right outer operator in $L^p(\M)$.

(iv) $\Rightarrow$ (i). Since $D^{\frac{1}{p}}\in [d\M]_p =[\M d]_p$, there are  sequences $(a_n)$ and $(b_n)$ in $\M$ such that $da_n\rightarrow D^{\frac{1}{p}}$ and $b_n d\rightarrow D^{\frac{1}{p}}$ in norm in $L^p(\M)$.  Using the continuity of $\E$,  we obtain that $d\E(a_n)\rightarrow D^{\frac{1}{p}}$ and $\E(b_n) d\rightarrow D^{\frac{1}{p}}$ in norm in $L^p(\D)$. Hence, we get the desired result.
\end{proof}
\begin{corollary}\label{cor:outer property-Lp(D)}
 Let $1\le p<\8$ and $0<r<\8$. If  $d\in L^p(\D)$ is an  outer operator and $rp\ge1$, then
    $|d|^\frac{1}{r}\in  L^{pr}(\D)$  is an  outer operator.
\end{corollary}
\begin{proof}
It is clear that  $|d|\in L^p(\D)$ is an  outer operator. Hence, by Theorem \ref{outer property-Lp(D)}, $|d|^\frac{1}{r}$ is an outer operator.
\end{proof}
\begin{corollary}\label{cor:outer property-Hp(A)}
 Let $1\le p<\8$ and  $d\in L^1(\D)^+$ be an  outer operator. If $0\le \eta\le1$, then
$$
H^{p}(\mathcal{A})=[d^{\frac{1-\eta}{p}}\A d^{\frac{\eta}{p}}]_p,\quad H^{p}_0(\mathcal{A})=[d^{\frac{1-\eta}{p}}\A_0 d^{\frac{\eta}{p}}]_p,\quad
L^{p}(\D)=[d^{\frac{1-\eta}{p}}\D d^{\frac{\eta}{p}}]_p
$$
and $L^{p}(\M)=[d^{\frac{1-\eta}{p}}\M d^{\frac{\eta}{p}}]_p$.
\end{corollary}

\begin{lemma}\label{outer property-diagnol}
 Let $1\le p<\8$, $1\le q,r<\8$ and $\frac{1}{p}-\frac{1}{r}=\frac{1}{q}$. If  $ d\in L^p(\D)$ is outer and $dD^{-\frac{1}{r}},\;D^{-\frac{1}{r}}d\in L^{q}(\M)$, then $dD^{-\frac{1}{r}},\;D^{-\frac{1}{r}}d\in L^{q}(\D)$  are outer operators.
\end{lemma}
\begin{proof} Since $dD^{-\frac{1}{r}}\in H^p(\A)D^{-\frac{1}{r}}\cap L^q(\M)$ and $dD^{-\frac{1}{r}}\in J(H^p(\A))D^{-\frac{1}{r}}\cap L^q(\M)$, by Lemma \ref{hp-spaces-carecterization}, we get $dD^{-\frac{1}{r}}\in H^q(\A)\cap J(H^q(\M)= L^q(\D)$. Similarly, $D^{-\frac{1}{r}}d\in L^{q}(\D)$. Using Theorem \ref{outer property-Lp(D)}, we obtain the desired result.
\end{proof}

\begin{lemma}\label{outer property}
 Let $1\le p<\8$, $1\le q,r<\8$ and $\frac{1}{p}+\frac{1}{r}=\frac{1}{q}$.
 \begin{enumerate}[\rm(i)]
   \item  If $ h\in H^p(\mathcal{A})$ is an outer operator, then
$hD^\frac{1}{r}$  and $D^{\frac{1}{r}}h\in H^{q}(\A)$ are  outer operators.
   \item  If $ d\in L^p(\D)$ is an outer operator, then
$dD^\frac{1}{r},\;D^{\frac{1}{r}}d\in L^{q}(\D)$ are  outer operators.
 \end{enumerate}
\end{lemma}
\begin{proof} (i) We only prove $hD^{\frac{1}{r}}$ is an  outer operator. A similar argument works for $D^{\frac{1}{r}}h$.  By \eqref{Hp-multiplication-Hr}, $[H^p(\A)D^{\frac{1}{r}}]_q=H^q(\A)$.
 We use same method as in the proof of (3) of Lemma \ref{analytic} to obtain that $[h\A_a]_p=[\A_ah]_p=H^p(\A)$. Hence, $[[h\A_a]_pD^{\frac{1}{r}}]_q=H^q(\A)$. Using Lemma \ref{analytic}, we get
 $$
 H^q(\A)=[[h\A_a]_pD^{\frac{1}{r}}]_q=[h\A_aD^{\frac{1}{r}}]_q=[hD^{\frac{1}{r}}\A_a]_q\subset[hD^{\frac{1}{r}}\A]_q\subset H^q(\A).
 $$
Thus $hD^{\frac{1}{r}}$ is a left outer operator. Similarly we can show $hD^{\frac{1}{r}}$ is a right outer operator.

(ii)  follows analogously.
\end{proof}

\begin{proposition}\label{condition outer} Let $1\le p<\8$ and    $h\in H^p(\A)$. Suppose that $\E(h)$ is an outer operator in $L^p(\D)$ and one of the the following conditions holds.
 \begin{enumerate}[\rm(i)]
   \item   $\frac{1}{p}-\frac{1}{r}=\frac{1}{2}\;(r>2)$ and  $hD^{-\frac{1}{r}},\;D^{-\frac{1}{r}}h\in L^{2}(\M)$;
   \item $\frac{1}{p}+\frac{1}{r}=\frac{1}{2}\;(r>2)$.
 \end{enumerate}
Then there is a left outer operator $g\in H^p(\A)$ and an isometry $u\in\A$ such that $h = ug$ (resp. there is a right outer operator $g'\in H^p(\A)$ and  $v\in\A$ such that $vv^*=1$ and $h=g'v$).
\end{proposition}

\begin{proof} First assume that condition (i) holds. By Lemma \ref{hp-spaces-carecterization}, we get $hD^{-\frac{1}{r}}\in H^{2}(\A)$.
Let $p'$ be the conjugate index of $p$.  Then for any $d\in\D$, we have that
$$
tr(\E(h)D^{-\frac{1}{r}}D^{\frac{1}{2}}d)=tr(\E(h)D^{\frac{1}{p'}}d)=tr(\E(hD^{\frac{1}{p'}}d))=tr(\E(hD^{-\frac{1}{r}})D^{\frac{1}{2}}d).
$$
By \eqref{equalityLp(D)}, we get
$$
tr(\E(h)D^{-\frac{1}{r}}f)=tr(\E(hD^{-\frac{1}{r}})f),\qquad \forall f\in L^2(\D).
$$
Hence, $\E(hD^{-\frac{1}{r}})=\E(h)D^{-\frac{1}{r}}$.  On the other hand, by Lemma \ref{outer property-diagnol}, $\E(h)D^{-\frac{1}{r}}$ is an outer operator in $L^2(\D)$.

We consider the orthogonal projection
$$
P:\;[hD^{-\frac{1}{r}}\A]_2\;\rightarrow\;[\E(hD^{-\frac{1}{r}})\D]_2.
$$
Then $P=\E|_{[hD^{-\frac{1}{r}}\A]_2}$ and $[\E(hD^{-\frac{1}{r}})\D]=[hD^{-\frac{1}{r}}\A]_2\ominus[hD^{-\frac{1}{r}}\A_0]_2$.
It follows that $\E(hD^{-\frac{1}{r}})$ is a cyclic separating vector for the wandering subspace $[\E(hD^{-\frac{1}{r}})\D]_2$ of $[hD^{-\frac{1}{r}}\A]_2$.
By Proposition \ref{cyclic and separating-L2}, there exists an isometry $u\in\M$ such that
$$
[hD^{-\frac{1}{r}}\A]_2=uH^2(\A).
$$
 We may write $hD^{-\frac{1}{r}}=uf$, for $f\in H^2(\A)$. Then
$$
[f\A]_2=u^*u[f\A]_2=u^*[hD^{-\frac{1}{r}}\A]_2=u^*uH^2(\A)=H^2(\A),
$$
i.e., $f$ is a left outer operator. On the other hand,
$$
0=tr(hD^{-\frac{1}{r}}aD^\frac{1}{2}b)=tr(u(faD^\frac{1}{2}b)),\qquad \forall a\in\A_0,\quad\forall b\in\A.
$$
Since $f$ is a left outer operator, by Proposition \ref{prop:left-right}, $[f\A_0]_2=H^2_0(\A)$. Hence, using \eqref{equalityHp} we obtain that $[f\A_0D^\frac{1}{2}\A]_1=H_0^1(\A)$.
It follows that $0=tr(ua)$ for any $a\in H_0^1(\A)$. By \eqref{charecterization A}, $u\in A$.  Let $g=fD^{\frac{1}{r}}$.  From the proof of Lemma \ref{outer property}, we know that $g$ is a left outer operator. This gives the desired result. Similarly, we  prove the
alternative claim.

If condition (ii) holds. The proof is similar to the above.
\end{proof}

\begin{lemma}\label{lem:equality} If $x\in L^2(\M)$ and  $u\in\M$ is a contraction such that $\|ux\|_2=\|x\|_2$, then $x=u^*ux$.
\end{lemma}

\begin{proof} We have that $x^*u^*ux\le x^*x$ and $tr(x^*u^*ux)=\|ux\|_2^2=\|x\|_2^2=tr(x^*x)$. Hence,
$$
\|x^*x-x^*u^*ux\|_1=tr(x^*x-x^*u^*ux)=0,
$$
so that  $x^*x=x^*u^*ux$.
Thus $\|(1-u^*u)^\frac{1}{2}x\|_2^2=\|x^*(1-u^*u)x\|_1=0$, therefore $(1-u^*u)x=(1-u^*u)^\frac{1}{2}[(1-u^*u)^\frac{1}{2}x]=0$, and $x=u^*ux$.
\end{proof}

In the finite case, $h\in H^2(\A)$ is a right outer operator if and only if  there is a cyclic separating vector for the right action $\D$ on the wandering subspace of $[h\A]_2$ and  $\|\E(h)\|_2=\|P(h)\|_2$, where $P$ is the orthogonal projection from $[h\A]_2$ to $[h\A]_2\ominus[h\A_0]_2$ (see \cite[Proposition 4.8]{BL1} or \cite[The remark after Theorem 4.4]{BL2}).
This result was extend to the case $1\le p<\8$ (see \cite[Theorem 4.4]{BL2}).

The following result  extends \cite[Theorem 4.4]{BL2} to the Haagerup
noncommutative  $H^{p}$-space case.

\begin{theorem}\label{thm:bilaterally-outer} Let $1\le p,r<\8$, and let    $h\in H^p(\A)$.
 \begin{enumerate}
 \item If $\frac{1}{p}+\frac{1}{r}=\frac{1}{2}$, then $h$ is an outer operator if and only if $\mathcal{E}(h)$ is  an outer operator in $L^p(\D)$ and $\|\E(hD^{\frac{1}{r}})\|_2=\|P(hD^{\frac{1}{r}})\|=\|P'(hD^{\frac{1}{r}})\|$, where $P$ is the orthogonal projection from $[hD^{\frac{1}{r}}\A]_2$ to $[hD^{\frac{1}{r}}\A]_2\ominus[hD^{\frac{1}{r}}\A_0]_2$ and
   $P'$ is the orthogonal projection from $[\A hD^{\frac{1}{r}}]_2$ to $[\A hD^{\frac{1}{r}}]_2\ominus[\A_0hD^{\frac{1}{r}}]_2$.
 \item Suppose that $\frac{1}{p}-\frac{1}{r}=\frac{1}{2}$ and $h D^{-\frac{1}{r}}\in L^2(\M)$. If $\mathcal{E}(h)$ is  an outer operator in $L^p(\D)$ and $\|\E(hD^{-\frac{1}{r}})\|_2=\|P(hD^{-\frac{1}{r}})\|=\|P'(hD^{-\frac{1}{r}})\|$, where $P$ is the orthogonal projection from $[hD^{-\frac{1}{r}}\A]_2$ to $[hD^{-\frac{1}{r}}\A]_2\ominus[hD^{-\frac{1}{r}}\A_0]_2$ and
   $P'$ is the orthogonal projection from $[\A hD^{-\frac{1}{r}}]_2$ to $[\A hD^{-\frac{1}{r}}]_2\ominus[\A_0hD^{-\frac{1}{r}}]_2$, then $h$ is an outer operator.
 \end{enumerate}
 \end{theorem}

\begin{proof} (i) $``\Rightarrow"$.  Using Proposition \ref{prop:left-right}, we obtain that $\mathcal{E}(h)$ is an outer operator in $L^p(\mathcal{D})$.
Since  $\E$ is a contractive projection from
 $H^2(\A)$ onto $L^2(\mathcal{D})$ with kernel $H^2_0(\A)$, we deduce that
 $$
 \|\E(hD^{\frac{1}{r}})\|_2=\inf_{h_0\in H^2_0(\A)}\|hD^{\frac{1}{r}}+h_0\|_2.
$$
 On the other hand, by Lemma \ref{outer property}, $hD^{\frac{1}{r}}$ is outer operator in $H^{2}(\A)$. Using Proposition \ref{prop:left-right}, we obtain that
 $$
 \begin{array}{rl}
 \|\E(hD^{\frac{1}{r}})\|_2  &=\inf_{h_0\in H^2_0(\A)}\|hD^{\frac{1}{r}}+h_0\|_2\\
 &  =\inf_{a_0\in \A_0}\|hD^{\frac{1}{r}}+hD^{\frac{1}{r}}a_0\|_2\\
 &=\|P(hD^{\frac{1}{r}})\|.
\end{array}
$$
Similarly, we can prove $ \|\E(hD^{\frac{1}{r}})\|_2=\|P'(hD^{\frac{1}{r}})\|$.

  $``\Leftarrow"$.  By Proposition \ref{condition outer}, $ h= ug$, where $g\in H^p(\A)$ is a left outer operator and $u \in \A$ is an isometry. On the other hand, it is clear that $gD^{\frac{1}{r}}$ is a left outer operator in $H^2(\A)$, Hence,
 $$
 \begin{array}{rl}
 \|\E(hD^{\frac{1}{r}})\|_2&=\|\E(u)\E(gD^{\frac{1}{r}})\|_2\le\|\E(gD^{\frac{1}{r}})\|_2\\
 &=\inf_{a_0\in A_0}\|gD^{\frac{1}{r}}+gD^{\frac{1}{r}}a_0\|_2\\
 &=\inf_{a_0\in A_0}\|u^*(hD^{\frac{1}{r}}+hD^{\frac{1}{r}}a_0)\|_2\\
 &\le\inf_{a_0\in A_0}\|hD^{\frac{1}{r}}+hD^{\frac{1}{r}}a_0\|_2\\
    & =\|P(hD^{\frac{1}{r}})\|=\|\E(hD^{\frac{1}{r}})\|_2.
\end{array}
$$
 This gives $\|\E(u)\E(gD^{\frac{1}{r}})\|_2=\|\E(gD^{\frac{1}{r}})\|_2$. Using Proposition \ref{prop:left-right}, we get $\E(gD^{\frac{1}{r}})$ is a left outer operator in in $L^2(\D)$, and so the left support of $\E(gD^{\frac{1}{r}})$ is 1. Applying Lemma \ref{lem:equality}, we obtain that $\E(u)$ is an isometry.
  On the other hand, we have that $\D\E(hD^{\frac{1}{r}})=\D\E(u)\E(gD^{\frac{1}{r}})\subset\D\E(gD^{\frac{1}{r}})$. Hence,
 $$
  L^2(\D)=[\D\E(hD^{\frac{1}{r}})]_2=[\D\E(u)\E(gD^{\frac{1}{r}})]_2\subset[\D\E(gD^{\frac{1}{r}})]_2\subset L^2(\D),
$$
i.e., $\E(gD^{\frac{1}{r}})$ is a right outer operator.  So, $\E(gD^{\frac{1}{r}})$ is an outer operator.
 From $\E(hD^{\frac{1}{r}})=\E(u)\E(gD^{\frac{1}{r}})$ follows that
 $$
 \E(u)\E(u^*)\E(hD^{\frac{1}{r}})=\E(u)\E(gD^{\frac{1}{r}})=\E(hD^{\frac{1}{r}}).
 $$
 Hence, $\E(u)\E(u^*)=1$, and so $\E(u)$ is a unitary. Therefore, $\E((u-\E(u))^*(u-\E(u)))=0$. So $u=\E(u)\in\D$ and $h$ is a left outer operator.

 Using the
alternative claim of Proposition \ref{condition outer} and the above method, we deduce that $h$ is a right outer operator.

(ii) From the proof of Proposition \ref{condition outer}, we know that $\E(hD^{-\frac{1}{r}})$ is an outer operator. Using same method as in the proof of (i), we obtain that $hD^{-\frac{1}{r}}$ is an outer operator in $H^2(\A)$. Hence, $h$ is an outer operator in $H^p(\A)$.
 \end{proof}

Let $d$ be a positive outer operator in $L^1(\D)$ with $\|d\|_1=1$. By Theorem \ref{outer property-Lp(D)}, $d$ is an invertible positive selfadjoint operator. Set
$$
\phi(x)=tr(xd),\qquad \forall x\in\M.
$$
It is clear that $\phi$ is a normal  faithful state on $\M$. Since $tr(\E(x))=tr(x)$ for $x\in L^1(\M)$ (see \cite[(2.4)]{JX}),  we get that
$$
\phi(\E(x))=tr(\E(x)d)=tr(\E(xd))=tr(xd)=\phi(x),\qquad \forall x\in\M.
$$
We denote the dual
weight   of $\phi$  by $\hat{\phi}$.   Then  $d$ is  the Radon-Nikodym derivative of  $\hat{\phi}$ with respect to
$\tau$ and
$$
\hat{\phi}(x)=\tau(xd),\qquad x\in\mathcal{N}_{+}.
$$
Hence, the role of $d$ is similar to that of $D$. It follows that if we replace $D$ by $d$ in Section 3 and 4, then the related results still hold.
\subsection*{Acknowledgment}
We thank the  referees for very useful comments.


\begin{thebibliography}{}

\bibitem{A} W. B. Arveson, \textit{Analyticity in operator algebras.}
Amer. J. Math. \textbf{89} (1967),  578--642.


\bibitem{BX} T. N. Bekjan and Q. Xu, \textit{Riesz and Szeg\"{o} type factorizations for  noncommutative Hardy spaces.}   J. Operator Theory \textbf{62} (2009),  215--231.


\bibitem{B1} T. N. Bekjan, \textit{Noncommutative Hardy space associated with semi-finite subdiagonal algebras.}  J. Math. Anal. Appl. \textbf{429} (2015), 1347--1369.

\bibitem{B2} T. N. Bekjan, \textit{Noncommutative symmetric Hardy spaces.} Integr. Equ. Oper. Theory  \textbf{81} (2015), 191--212.

\bibitem{BM} T. N. Bekjan and M. Raikhan, \textit{Interpolation of Haagerup noncommutative Hardy spaces.}  Banach J. Math. Anal. \textbf{13}(2019), 798--814.

\bibitem{BL3} D. P. Blecher and L. E. Labuschagne, \textit{A Beurling theorem for noncommutative $L^{p}$.}
 J. Operator Theory \textbf{59} (2008), 29--51.
\bibitem{BL1} D. P. Blecher and L. E. Labuschagne, \textit{Applications of the Fuglede-Kadison
determinant: Szeg$\ddot{o}$s theorem  and outers for
noncommutatuve $ H_{p}$.}  Trans. Amer. Math. Soc.  \textbf{360} (2008), 6131--6147.

\bibitem{BL2} D. P. Blecher and L. E. Labuschagne, \textit{ Outers for noncommutative $H^p$ revisited.}
Studia Math. \textbf{217} (2013),  265--287.

\bibitem{BL4} D. P. Blecher and L. E. Labuschagne,  \textit{Von Neumann algebraic Hp theory.}  Function Spaces, Contemporary Mathematics {\bf435}, 89--114,  Amer. Math.  Soc. , Providence, RI, 2007.

\bibitem{CHS} Y. Chen, D. Hadwin and J. Shen, \textit{A non-commutative Beurling's theorem with respect to unitarily invariant
norms.} J. Operator Theory \textbf{75} (2016), 497--523.



\bibitem{H1}   U. Haagerup, \textit{`$L^{p}$- spaces associated wth an arbitrary von Neumann algegra' in
Alg\`{e}bres d'op\'{e}rateurs et leurs applications en physique math\'{e}matique.} (
Proc. Colloq., Marseille, 1977),  Colloq. Internat. CNRS  \textbf{274} (CNRS, Paris, 1979),  175--184.


\bibitem{JOS}  G. Ji, T. Ohwada and K.-S. Saito, \textit{Certain structure of subdiagonal algebras.} J.  Operator
Theory \textbf{39} (1998), 309--317.

\bibitem{J1}  G. Ji,  \textit{A noncommutative version of $H^p$ and characterizations of subdiagonal algebras.} Integr. Equ. Oper. Theory \textbf{72} (2012), 183--191.

\bibitem{J2}  G. Ji,  \textit{Analytic Toeplitz algebras and the Hilbert transform
associated with a subdiagonal algebra.} Sci. China Math. \textbf{57} (2014), 579--588.


\bibitem{JS} M. Junge and D. Sherman, \textit{Noncommutative $L^p$-modules.} J. Operator Theory \textbf{53} (2005),
3--34.

\bibitem{JX} M. Junge and Q. Xu,  \textit{Noncommutative Burkholder/Rosenthal inequalities.}  Ann.
Probab. {\bf 31}  (2003), 948--995.


\bibitem{L2} L. E. Labuschagne, \textit{Invariant subbspaces for $H^2$ spaces of $\sigma$-finite algebras.}  Bull.
London Math. Soc. \textbf{49} (2017), 33--44.

\bibitem{NW} T. Nakazi and Y. Watatani, \textit{Invariant subspace theorems for subdiagonal algebras.} J. Operator Theory \textbf{37} (1997), 379--395.

\bibitem{PT} G. K. Pedersen and  M. Takesaki, \textit{The Radon-Nikodym theorem for von Neuman algebras.}  Acta Math.  \textbf{130} (1973),
53--87.

\bibitem{PX}  G. Pisier and Q. Xu, \textit{Noncommutative $L^{p}$-spaces.} In: Handbook of the geometry of Banach spaces, Vol. 2.  North-Holland, Amsterdam, 2003, 1459--1517.


\bibitem{Sa} L. Sager, \textit{A Beurling-Blecher-Labuschagne theorem for noncommutative Hardy spaces associated with semifinite
von Neumann algebras.}  Integr. Equ. Op. Theory  \textbf{86} (2016), 377--407.

\bibitem{SL} L. Sager  and W. Liu,  \textit{A Beurling-Chen-Hadwin-Shen theorem for
noncommutative Hardy spaces associated with
semifinite von Neumann algebras with unitarily
invariant norms.} J. Operator Theory \textbf{82} (2019),
49--78.


\bibitem{SW} T. P. Srinivasan and J. -K. Wang,  \textit{Weak*-Dirichlet algebras.} In: Function algebras (ed. F. T. Birtel), Scott
Foresman and Co., Chicago, 1966, 216--249.

\bibitem{T3}  M. Takesaki, \textit{Duality for crossed products and the structure of von Neumann algebras  of type {\rm III}.}
Acta Math. \textbf{131}(1973), 249--310.

 \bibitem{Te}  M. Terp, \textit{$L^{p}$-spaces associated wth an arbitrary von Neumann algegras.}  Notes.
 Math. Institute, Copenhagen Univ., 1981.

 \bibitem{X} Q. Xu, \textit{On the maximality of subdiagonal algebras.} J. Operator Theory \textbf{54} (2005), 137--146.
\end{thebibliography}
\end{document}